\documentclass[10pt]{article}
\usepackage{lmodern}
\usepackage{amsmath}
\usepackage[T1]{fontenc}
\usepackage[utf8]{inputenc}
\usepackage{authblk}
\usepackage{amsfonts}
\usepackage{graphicx}
\usepackage{rotating}
\usepackage{amssymb}
\usepackage[english]{babel}
\usepackage{xcolor}
\usepackage{amsthm}
\usepackage{graphicx}
\usepackage{mathrsfs}
\usepackage{makecell}
\usepackage{microtype}
\usepackage{array}
\usepackage{multirow}
\usepackage{enumerate}
\IfFileExists{BOONDOX-r-calo.tfm}
{\usepackage[cal=boondoxo,bb=ams]{mathalfa}}{}
\usepackage{hyperref}
\usepackage{booktabs}
\hypersetup{hidelinks}

\newtheorem{theorem}{Theorem}[section]
\newtheorem{prop}{Proposition}[section]
\newtheorem{lemma}{Lemma}[section]

\newtheorem{remark}{Remark}[section]

\newcommand{\ml}{\mathcal}
\newcommand{\mb}{\mathbb}

\DeclareMathOperator{\nlin}{nlin}

\def\XXint#1#2#3{{\setbox0=\hbox{$#1{#2#3}{\int}$ }
		\vcenter{\hbox{$#2#3$ }}\kern-.6\wd0}}

\title{Global in-time behavior for the semilinear nonlocal heat exchanger system with different diffusion orders}

\author[2]{Yan Liu\thanks{Yan Liu (ly801221@163.com)}}
\affil[2]{Department of Applied Mathematics, Guangdong University of Finance,\authorcr 510521 Guangzhou, China}

\date{}

\begin{document}
		\maketitle

\begin{abstract}
	\medskip
	We consider a semilinear nonlocal heat exchanger system whose two equations may have different fractional diffusion orders. Although these orders need not coincide, both components have the same leading diffusion for large time. Its order is $\sigma_*:=\min\{\sigma_1,\sigma_2\}$, while its coefficient depends on which equation contains the lower order. If $\sigma_1\neq\sigma_2$, we further estimate the difference between the linear solution and the effective fractional heat flow. The additional decay changes at the threshold $\sigma^*=2\sigma_*$, where $\sigma^*:=\max\{\sigma_1,\sigma_2\}$. Based on these linear estimates, we prove global in-time existence for small data if $\min\{p,q\}>1+\frac{2\sigma_*}{n}$, derive the corresponding $L^m$ asymptotic profiles, and obtain lower bounds for the lifespan in the sub-critical and critical cases.
	
	\noindent\textbf{AMS Classification (2020)} 35K57, 35R11, 35B40, 35A01
\end{abstract}

\fontsize{12}{15}
\selectfont
 
\section{Introduction}\label{Section_Introduction}\setcounter{equation}{0}
\hspace{5mm}In this manuscript, we consider the semilinear nonlocal heat exchanger system in the whole space $\mb{R}^n$ with any dimension $n\geqslant 1$, namely,
\begin{align}\label{Eq-Semi-Nonlocal-Heat-System}
\begin{cases}
u_t+(-\Delta)^{\sigma_1}u+\mu u-\nu v=u^p,&x\in\mb{R}^n,\ t>0,\\
v_t+(-\Delta)^{\sigma_2}v-\mu u+\nu v=v^q,&x\in\mb{R}^n,\ t>0,\\
u(0,x)=\varepsilon u_0(x),\ v(0,x)=\varepsilon v_0(x),&x\in\mb{R}^n,
\end{cases}
\end{align}
with $\sigma_1,\sigma_2>0$, the constants $\mu,\nu>0$, and $p,q>1$, where the parameter $\varepsilon>0$ denotes the size of initial data. All functions are assumed to be real-valued, and throughout this manuscript the notation $w^r$ in a nonlinear term is understood as $|w|^r$ for $r>1$. The fractional Laplacians are defined via
\begin{align*}
(-\Delta)^{\sigma_j}f:=\ml{F}^{-1}\left(|\xi|^{2\sigma_j}\widehat{f}\,\right)\ \ \mbox{with} \ \  j\in\{1,2\}.
\end{align*}
For later use, we put
\begin{align*}
\sigma_*:=\min\{\sigma_1,\sigma_2\},\qquad \sigma^*:=\max\{\sigma_1,\sigma_2\}.
\end{align*}
Our main purpose is to derive large time asymptotic profiles of global in-time solutions in the super-critical case
\begin{align*}
\min\{p,q\}>p_{\mathrm{Fuj}}\left(\frac{n}{\sigma_*}\right):=1+\frac{2\sigma_*}{n}.
\end{align*}
As a byproduct, lower bound estimates of the lifespan with respect to $\varepsilon$ will be given in the sub-critical and critical cases. Here and in what follows, $T_{\varepsilon}$ denotes the maximal existence time of the unique mild solution in the solution class introduced in Section \ref{Section_Existence}.

Taking $\sigma_1=\sigma_2=1$ in \eqref{Eq-Semi-Nonlocal-Heat-System}, it becomes a particular case of the Fujita-type semilinear local heat exchanger system from \cite{Treton=2024}, which is motivated by the field-road reaction-diffusion system \cite{Berestycki-Roquejoffre-Rossi=2013,Alfaro-Ducasse-Treton=2023}. The author of \cite{Treton=2024} derived global in-time existence for sufficiently small data if $\min\{p,q\}>1+\frac{2}{n}$, and finite time blow-up if $\min\{p,q\}<1+\frac{2}{n}$ for non-both-trivial, non-negative initial data. The case $\sigma_1=\sigma_2=\sigma$ is the fractional counterpart of this model. In that case, the fractional heat operator commutes with the constant exchange matrix, and the two spectral modes are explicitly given by
\begin{align*}
\mathrm{e}^{-t|\xi|^{2\sigma}}\ \ \mbox{and}\ \ \mathrm{e}^{-t(|\xi|^{2\sigma}+\mu+\nu)}.
\end{align*}

The corresponding linearized model to \eqref{Eq-Semi-Nonlocal-Heat-System} with $\varepsilon=1$ is
\begin{align}\label{Eq-Linear-Heat-Exchanger-System}
\begin{cases}
\bar{u}_t+(-\Delta)^{\sigma_1}\bar{u}+\mu \bar{u}-\nu \bar{v}=0,&x\in\mb{R}^n,\ t>0,\\
\bar{v}_t+(-\Delta)^{\sigma_2}\bar{v}-\mu \bar{u}+\nu \bar{v}=0,&x\in\mb{R}^n,\ t>0,\\
\bar{u}(0,x)=\bar{u}_0(x),\ \bar{v}(0,x)=\bar{v}_0(x),&x\in\mb{R}^n.
\end{cases}
\end{align}
The coupling terms describe the exchange between the two components. When $\sigma_1\neq\sigma_2$, the diffusion matrix and the exchange matrix do not commute. Consequently, the persistent mode is no longer an exact fractional heat flow. Its low-frequency eigenvalue satisfies
\begin{align*}
\lambda_-(\xi)=\frac{\nu}{\mu+\nu}|\xi|^{2\sigma_1}
+\frac{\mu}{\mu+\nu}|\xi|^{2\sigma_2}
+O\left(\big(|\xi|^{2\sigma_1}-|\xi|^{2\sigma_2}\big)^2\right)
\end{align*}
as $|\xi|\to0$. Thus, the lower order $\sigma_*$ determines the common large time decay of both components, while the coefficient of the effective diffusion depends on whether $\sigma_*=\sigma_1$ or $\sigma_*=\sigma_2$.

Let us briefly recall the scalar semilinear fractional heat equation
\begin{align}\label{Eq-Single-Fractional}
\begin{cases}
U_t+(-\Delta)^{\sigma}U=U^p,&x\in\mb{R}^n,\ t>0,\\
U(0,x)=\varepsilon U_0(x),&x\in\mb{R}^n,
\end{cases}
\end{align}
with $\sigma>0$. The classical works \cite{Sugitani=1975,Gal-Pohoza=2002,Ishige-Kawakami-Kobayashi=2014,Hisa-Ishige=2018} and references therein identify $1+\frac{2\sigma}{n}$ as the Fujita exponent. Under the corresponding assumptions on the initial data, its small-data lifespan satisfies
\begin{align*}
\begin{cases}
T_{\varepsilon}^U\approx \displaystyle{\varepsilon^{-\frac{p-1}{1-\frac{n}{2\sigma}(p-1)}}},&p<1+\frac{2\sigma}{n},\\
\exp\left(c\varepsilon^{-\frac{2\sigma}{n}}\right)\leqslant T_{\varepsilon}^U
\leqslant\exp\left(C\varepsilon^{-\frac{2\sigma}{n}}\right),&p=1+\frac{2\sigma}{n},\\
T_{\varepsilon}^U=+\infty\ \ \mbox{for sufficiently small data},&p>1+\frac{2\sigma}{n}.
\end{cases}
\end{align*}
The relation $f\approx g$ means $g\lesssim f\lesssim g$, with constants independent of $\varepsilon$.

The exchange terms couple the two diffusion orders in the low-frequency regime. 
After diagonalizing the Fourier matrix of \eqref{Eq-Linear-Heat-Exchanger-System}, we derive uniform $L^r$-$L^m$ estimates for the solution operators. 
If $\sigma_1=\sigma_2$, the transient mode yields an exponentially decaying contribution. 
For $\sigma_1\neq\sigma_2$, the transient mode remains exponentially damped, whereas the difference between the persistent mode and the effective fractional heat flow has an algebraic decay. 
The corresponding decay factor depends on the relation between $\sigma^*$ and $2\sigma_*$.


\section{Main results}\label{Section_Results}\setcounter{equation}{0}
\hspace{5mm}We first introduce the effective fractional diffusion. Set
\begin{align}\label{Def-effective}
d_*:=
\begin{cases}
\displaystyle\frac{\nu}{\mu+\nu},&\sigma_1<\sigma_2,\\[2mm]
1,&\sigma_1=\sigma_2,\\[1mm]
\displaystyle\frac{\mu}{\mu+\nu},&\sigma_2<\sigma_1,
\end{cases}
\qquad
G_*(t,x):=\ml{F}^{-1}_{\xi\to x}\left(\mathrm{e}^{-d_*t|\xi|^{2\sigma_*}}\right).
\end{align}
The associated convolution operator is denoted by $\ml{G}_*(t,|D|)$. If $\sigma_1\neq\sigma_2$, we additionally put
\begin{align}\label{Def-theta}
\rho_*:=\min\{\sigma^*,2\sigma_*\},\qquad
\vartheta_*:=\frac{\rho_*-\sigma_*}{\sigma_*}
=\min\left\{\frac{\sigma^*-\sigma_*}{\sigma_*},1\right\}>0.
\end{align}

To define the solution operators, let
\begin{align}\label{Def-matrix-symbol}
\mb{A}(\xi):=
\begin{pmatrix}
|\xi|^{2\sigma_1}+\mu&-\nu\\
-\mu&|\xi|^{2\sigma_2}+\nu
\end{pmatrix}
\end{align}
and write
\begin{align}\label{Def-solution-operators}
\mathrm{e}^{-t\mb{A}(\xi)}
=
\begin{pmatrix}
\widehat{K_0^u}(t,\xi)&\widehat{K_1^u}(t,\xi)\\
\widehat{K_0^v}(t,\xi)&\widehat{K_1^v}(t,\xi)
\end{pmatrix}.
\end{align}
The notation $K_j^w(t,|D|)$ stands for the corresponding Fourier multiplier.

\begin{theorem}\label{Thm=Global-Solution}
Let $u_0,v_0\in L^{\infty}\cap L^1$ and
\begin{align*}
\min\{p,q\}>p_{\mathrm{Fuj}}\left(\frac{n}{\sigma_*}\right).
\end{align*}
Then, there exists $\varepsilon_0>0$ such that for any $\varepsilon\in(0,\varepsilon_0]$ the semilinear nonlocal heat exchanger system \eqref{Eq-Semi-Nonlocal-Heat-System} admits unique global in-time mild solutions satisfying
\begin{align*}
(u,v)\in\ml{C}([0,+\infty),L^{m_1})\times\ml{C}([0,+\infty),L^{m_2})
\end{align*}
for any $1\leqslant m_1,m_2<+\infty$, together with
\begin{align*}
u,v\in L^{\infty}([0,+\infty),L^{\infty})\cap\ml{C}((0,+\infty),L^{\infty}).
\end{align*}
Furthermore, for $w\in\{u,v\}$, they satisfy
\begin{align}
w(t,x)&=\varepsilon\big(K_0^w(t,|D|)u_0(x)+K_1^w(t,|D|)v_0(x)\big)\notag\\
&\quad+\int_0^t\big(K_0^w(t-\tau,|D|)[u(\tau,x)]^p
+K_1^w(t-\tau,|D|)[v(\tau,x)]^q\big)\,\mathrm{d}\tau,\label{Equa-01}
\end{align}
and
\begin{align}
\|w(t,\cdot)\|_{L^{m}}
\lesssim\varepsilon(1+t)^{-\frac{n}{2\sigma_*}(1-\frac{1}{m})}
\|(u_0,v_0)\|_{(L^{\infty}\cap L^1)^2}\label{Est-03}
\end{align}
for every $1\leqslant m\leqslant+\infty$.
\end{theorem}

\begin{remark}
When $\sigma_1=\sigma_2=\sigma$, the effective coefficient in \eqref{Def-effective} is $d_*=1$, and Theorem \ref{Thm=Global-Solution} reduces to the equal-order result. In particular, for $\sigma=1$, its exponent condition and $L^{\infty}$ decay agree with the existence part of \cite[Theorem 2.3]{Treton=2024}. If $\sigma_1\neq\sigma_2$, both components still have the common decay order $\sigma_*$, although their uncoupled diffusion orders are different.
\end{remark}

We next state the large time profiles. For an integrable function, let
\begin{align*}
P_f:=\int_{\mb{R}^n}f(x)\,\mathrm{d}x,\qquad
\ml{P}_f:=\int_0^{+\infty}\int_{\mb{R}^n}f(t,x)\,\mathrm{d}x\,\mathrm{d}t.
\end{align*}
\begin{theorem}\label{Thm=Large-time}
Let the assumptions in Theorem \ref{Thm=Global-Solution} hold. Then, the global in-time small data solutions satisfy
\begin{align*}
\lim\limits_{t\to+\infty}
t^{\frac{n}{2\sigma_*}(1-\frac{1}{m})}
\left\|w(t,\cdot)-\frac{\gamma}{\mu+\nu}G_*(t,\cdot)
\left(\varepsilon P_{u_0+v_0}+\ml{P}_{u^p+v^q}\right)\right\|_{L^m}=0
\end{align*}
for every $1\leqslant m\leqslant+\infty$, where $\gamma=\nu$ if $w=u$, and $\gamma=\mu$ if $w=v$. Particularly, if
\begin{align*}
\varepsilon P_{u_0+v_0}+\ml{P}_{u^p+v^q}\neq0,
\end{align*}
then
\begin{align*}
t^{-\frac{n}{4\sigma_*}}
\left|\varepsilon P_{u_0+v_0}+\ml{P}_{u^p+v^q}\right|
\lesssim\|w(t,\cdot)\|_{L^2}
\lesssim\varepsilon t^{-\frac{n}{4\sigma_*}}
\|(u_0,v_0)\|_{(L^{\infty}\cap L^1)^2}
\end{align*}
as large time $t\gg1$.
\end{theorem}

\begin{remark}
The vector multiplying the effective fractional heat kernel is
\begin{align*}
\frac{1}{\mu+\nu}(\nu,\mu)^{\mathrm{T}}.
\end{align*}
This vector is independent of $\sigma_1$ and $\sigma_2$, since it is the null eigenvector of the exchange matrix at $\xi=0$. On the other hand, $d_*$ depends on which equation carries the lower diffusion order: it is $\frac{\nu}{\mu+\nu}$ if $\sigma_1<\sigma_2$, and $\frac{\mu}{\mu+\nu}$ if $\sigma_2<\sigma_1$.
\end{remark}

\begin{remark}
	The mass appearing in the leading profile in Theorem \ref{Thm=Large-time} can be read directly from the system. After adding the two equations in \eqref{Eq-Semi-Nonlocal-Heat-System}, the exchange terms cancel, whereas the fractional diffusion terms vanish at $\xi=0$. Hence,
	\begin{align*}
		P_{u(t,\cdot)+v(t,\cdot)}
		=\varepsilon P_{u_0+v_0}
		+\int_0^tP_{u^p(\tau,\cdot)+v^q(\tau,\cdot)},\mathrm{d}\tau.
	\end{align*}
	The space-time integrability of the nonlinear terms gives
	\begin{align*}
		\lim\limits_{t\to+\infty}P_{u(t,\cdot)+v(t,\cdot)}
		=\varepsilon P_{u_0+v_0}+\ml{P}_{u^p+v^q}.
	\end{align*}
	Thus, the coefficient in the asymptotic profile is the limiting total mass of the two components. In particular, it is positive when $u_0$ and $v_0$ are non-negative and not both trivial. Therefore, the non-degeneracy condition in Theorem \ref{Thm=Large-time} is automatically satisfied in this case.
\end{remark}

\begin{remark}\label{Remark-Blow-up}
Theorem \ref{Thm=Global-Solution} gives global in-time existence in the super-critical case. We expect finite time blow-up in the remaining cases under suitable non-triviality and sign assumptions on the initial data. For $\sigma_1=\sigma_2=1$, \cite[Theorem 2.4]{Treton=2024} proves systematic blow-up in the sub-critical case for non-both-trivial, non-negative initial data. Moreover, if $\sigma_1=\sigma_2=\sigma$ and $p=q$, summing the two equations gives
\begin{align*}
(u+v)_t+(-\Delta)^{\sigma}(u+v)
=u^p+v^p\geqslant2^{1-p}|u+v|^p.
\end{align*}
When $\sigma_1\neq\sigma_2$, the sum contains
$(-\Delta)^{\sigma_1}u+(-\Delta)^{\sigma_2}v$ and does not reduce to a scalar fractional heat equation. A corresponding blow-up theorem is not proved here.
\end{remark}

Motivated by Remark \ref{Remark-Blow-up}, we derive lower bound estimates for the maximal existence time. They remain valid whether $T_\varepsilon$ is finite or infinite.
\begin{theorem}\label{Thm=lower-lifespan}
Let $u_0,v_0\in L^{\infty}\cap L^1$ and $\sigma_1,\sigma_2>0$. For every $\varepsilon>0$, the semilinear problem \eqref{Eq-Semi-Nonlocal-Heat-System} admits unique local in-time mild solutions such that
\begin{align*}
(u,v)\in\ml{C}([0,T],L^{m_1})\times\ml{C}([0,T],L^{m_2})
\end{align*}
for any $1\leqslant m_1,m_2<+\infty$, and
\begin{align*}
u,v\in L^{\infty}([0,T],L^{\infty})\cap\ml{C}((0,T],L^{\infty})
\end{align*}
for some $T>0$. Moreover, if
\begin{align*}
\min\{p,q\}\leqslant p_{\mathrm{Fuj}}\left(\frac{n}{\sigma_*}\right),
\end{align*}
then there exist constants $c>0$ and $\varepsilon_1>0$, independent of $\varepsilon$, such that for any $\varepsilon\in(0,\varepsilon_1]$ the solution exists on every interval $[0,T]$ with
\begin{align*}
0<T\leqslant
\begin{cases}
c\displaystyle{\varepsilon^{-\frac{\min\{p,q\}-1}
{1-\frac{n}{2\sigma_*}(\min\{p,q\}-1)}}},
&\min\{p,q\}<p_{\mathrm{Fuj}}\left(\frac{n}{\sigma_*}\right),\\[3mm]
\exp\left(c\varepsilon^{-\frac{2\sigma_*}{n}}\right),
&\min\{p,q\}=p_{\mathrm{Fuj}}\left(\frac{n}{\sigma_*}\right).
\end{cases}
\end{align*}
The estimates \eqref{Est-03} hold for any $t\in[0,T]$. Consequently,
\begin{align*}
T_\varepsilon\geqslant
\begin{cases}
c\displaystyle{\varepsilon^{-\frac{\min\{p,q\}-1}
{1-\frac{n}{2\sigma_*}(\min\{p,q\}-1)}}},
&\min\{p,q\}<p_{\mathrm{Fuj}}\left(\frac{n}{\sigma_*}\right),\\[3mm]
\exp\left(c\varepsilon^{-\frac{2\sigma_*}{n}}\right),
&\min\{p,q\}=p_{\mathrm{Fuj}}\left(\frac{n}{\sigma_*}\right).
\end{cases}
\end{align*}
\end{theorem}

\begin{remark}
The exponent in Theorem \ref{Thm=lower-lifespan} is the one associated with the effective order $\sigma_*$. The sharpness of these lower bounds for the different-order system would require matching upper bounds, which are not proved in this manuscript.
\end{remark}

\begin{remark}
The same approach applies to more general nonlinearities $(f_1(u,v),f_2(u,v))^{\mathrm{T}}$ under the corresponding growth and local Lipschitz assumptions. For example, one may consider additive nonlinearities $f_j(u,v)=u^{p_j}+v^{q_j}$ or multiplicative nonlinearities $f_j(u,v)=u^{p_j}v^{q_j}$ for $j\in\{1,2\}$. Nonlinearities involving moduli of continuity, such as $(u^p\ml{M}_1(|u|),v^q\ml{M}_2(|v|))^{\mathrm{T}}$, may be treated in a similar way under suitable assumptions on $\ml{M}_1$ and $\ml{M}_2$; see \cite{Ebert-Girardi-Reissig=2020}.
\end{remark}

\section{Linear nonlocal heat exchanger system}\label{Section_Linear}\setcounter{equation}{0}
\subsection{Preliminaries in the Fourier space}
\hspace{5mm}Applying the partial Fourier transform with respect to the spatial variable to \eqref{Eq-Linear-Heat-Exchanger-System}, we obtain
\begin{align}\label{Eq-Fourier-Imagine}
\begin{cases}
\widehat{\bar{u}}_t+(|\xi|^{2\sigma_1}+\mu)\widehat{\bar{u}}-\nu\widehat{\bar{v}}=0,
&\xi\in\mb{R}^n,\ t>0,\\
\widehat{\bar{v}}_t-\mu\widehat{\bar{u}}+(|\xi|^{2\sigma_2}+\nu)\widehat{\bar{v}}=0,
&\xi\in\mb{R}^n,\ t>0,\\
\widehat{\bar{u}}(0,\xi)=\widehat{\bar{u}}_0(\xi),\
\widehat{\bar{v}}(0,\xi)=\widehat{\bar{v}}_0(\xi),
&\xi\in\mb{R}^n.
\end{cases}
\end{align}
For simplicity, let
\begin{align*}
a=a(\xi):=|\xi|^{2\sigma_1},\qquad
b=b(\xi):=|\xi|^{2\sigma_2},\qquad
S:=\mu+\nu.
\end{align*}
The reduction methodology used, for example, in
\cite{Chen-Ikehata=2023,Chen-Liu=2023,Aslan-Chen=2025,Chen-Takeda=2023,Chen-Liu=2025,Chen-Ikehata=2027,Chen=2026,Chen-Ikehata=2024}
still applies. Acting with $\partial_t+b+\nu$ on the first equation in \eqref{Eq-Fourier-Imagine}, and with $\partial_t+a+\mu$ on the second one, shows that both unknowns satisfy
\begin{align}\label{Eq-reduced}
\widehat{\bar{w}}_{tt}+(a+b+S)\widehat{\bar{w}}_t
+(ab+\nu a+\mu b)\widehat{\bar{w}}=0.
\end{align}
Indeed, the first component satisfies
\begin{align*}
(\partial_t+b+\nu)(\partial_t+a+\mu)\widehat{\bar{u}}
-\mu\nu\widehat{\bar{u}}=0,
\end{align*}
because $(\partial_t+b+\nu)\widehat{\bar{v}}=\mu\widehat{\bar{u}}$ by the second equation. Expanding the two factors gives \eqref{Eq-reduced}, and the same calculation with the roles of the two equations interchanged applies to $\widehat{\bar{v}}$.
The initial time derivatives are
\begin{align*}
\widehat{\bar{u}}_t(0,\xi)
&=-(a+\mu)\widehat{\bar{u}}_0(\xi)+\nu\widehat{\bar{v}}_0(\xi),\\
\widehat{\bar{v}}_t(0,\xi)
&=\mu\widehat{\bar{u}}_0(\xi)-(b+\nu)\widehat{\bar{v}}_0(\xi).
\end{align*}

The characteristic roots of \eqref{Eq-reduced} are $-\lambda_-(\xi)$ and $-\lambda_+(\xi)$, where
\begin{align}\label{Def-eigenvalues}
\lambda_{\pm}(\xi)
:=\frac{a+b+S\pm\delta(\xi)}{2},\qquad
\delta(\xi):=\sqrt{(a-b+\mu-\nu)^2+4\mu\nu}.
\end{align}
Equivalently, with the matrix $\mb{A}(\xi)$ from \eqref{Def-matrix-symbol},
\begin{align}\label{Eq-spectral-decomposition}
\mathrm{e}^{-t\mb{A}(\xi)}
=\mathrm{e}^{-t\lambda_-(\xi)}\mb{P}_-(\xi)
+\mathrm{e}^{-t\lambda_+(\xi)}\mb{P}_+(\xi),
\end{align}
where
\begin{align}\label{Def-projectors}
\mb{P}_-(\xi):=\frac{\lambda_+(\xi)\mb{I}-\mb{A}(\xi)}{\delta(\xi)},
\qquad
\mb{P}_+(\xi):=\frac{\mb{A}(\xi)-\lambda_-(\xi)\mb{I}}{\delta(\xi)}.
\end{align}
In particular,
\begin{align}\label{Def-P0}
\mb{P}_0:=\mb{P}_-(0)
=\frac{1}{\mu+\nu}
\begin{pmatrix}
\nu&\nu\\
\mu&\mu
\end{pmatrix}.
\end{align}

The matrix $\mb{A}(\xi)$ is similar to the symmetric non-negative matrix
\begin{align*}
\begin{pmatrix}
a+\mu&-\sqrt{\mu\nu}\\
-\sqrt{\mu\nu}&b+\nu
\end{pmatrix}
\end{align*}
through the constant transformation $\operatorname{diag}(\sqrt{\nu},\sqrt{\mu})$. 
More precisely, with $\mb{D}:=\operatorname{diag}(\sqrt{\nu},\sqrt{\mu})$, one has
\begin{align*}
\mb{D}^{-1}\mb{A}(\xi)\mb{D}
=\begin{pmatrix}
a+\mu&-\sqrt{\mu\nu}\\
-\sqrt{\mu\nu}&b+\nu
\end{pmatrix}.
\end{align*}
Hence,
$0\leqslant\lambda_-(\xi)<\lambda_+(\xi)$, with strict positivity of $\lambda_-(\xi)$ for $\xi\neq0$. Moreover,
\begin{align}\label{Est-eigenvalues}
\lambda_-(\xi)\approx|\xi|^{2\sigma_*},
\qquad
\lambda_+(\xi)\approx1+|\xi|^{2\sigma^*}.
\end{align}
Indeed,
\begin{align*}
\lambda_-(\xi)\lambda_+(\xi)=ab+\nu a+\mu b,
\qquad
\lambda_-(\xi)+\lambda_+(\xi)=a+b+S,
\end{align*}
and \eqref{Est-eigenvalues} follows by considering $|\xi|\leqslant1$ and $|\xi|\geqslant1$ separately.
For completeness, if $|\xi|\leqslant1$, then
\begin{align*}
\lambda_+(\xi)\approx1,
\qquad
\lambda_-(\xi)\lambda_+(\xi)
=ab+\nu a+\mu b\approx|\xi|^{2\sigma_*},
\end{align*}
which proves the desired estimate for $\lambda_-$. If $|\xi|\geqslant1$, then
\begin{align*}
\lambda_-(\xi)+\lambda_+(\xi)&\approx|\xi|^{2\sigma^*},\\
\lambda_-(\xi)\lambda_+(\xi)&\approx
|\xi|^{2(\sigma_*+\sigma^*)}.
\end{align*}
Since
\begin{align*}
\frac{1}{2}\big(a+b+S\big)\leqslant\lambda_+(\xi)
\leqslant a+b+S,
\end{align*}
it follows first that $\lambda_+(\xi)\approx|\xi|^{2\sigma^*}$ and then, from the product identity, that $\lambda_-(\xi)\approx|\xi|^{2\sigma_*}$.

We next record the low-frequency expansion which distinguishes the three relations between $\sigma_1$ and $\sigma_2$.
\begin{lemma}\label{Lemma-spectral-expansion}
As $|\xi|\to0$, it holds that
\begin{align}\label{Exp-lambda}
\lambda_-(\xi)
=\frac{\nu a+\mu b}{S}
-\frac{\mu\nu}{S^3}(a-b)^2
+O\big((a+b)^3\big),
\end{align}
and
\begin{align}\label{Exp-projector}
\mb{P}_-(\xi)=\mb{P}_0+O(a+b).
\end{align}
If $\sigma_1\neq\sigma_2$, then
\begin{align}\label{Exp-effective}
\lambda_-(\xi)=d_*|\xi|^{2\sigma_*}+O(|\xi|^{2\rho_*}).
\end{align}
More precisely, for every multi-index $\alpha$ and $0<|\xi|\leqslant1$,
\begin{align}\label{Est-symbol-low}
\left|\partial_\xi^\alpha
\big(\lambda_-(\xi)-d_*|\xi|^{2\sigma_*}\big)\right|
&\lesssim|\xi|^{2\rho_*-|\alpha|},\\
\left|\partial_\xi^\alpha
\big(\mb{P}_-(\xi)-\mb{P}_0\big)\right|
&\lesssim|\xi|^{2\sigma_*-|\alpha|}.
\end{align}
\end{lemma}
\begin{proof}
Since $\delta(0)=S>0$, the square root in \eqref{Def-eigenvalues} is analytic with respect to $a-b$ in a neighborhood of the origin. Taylor's formula gives
\begin{align*}
\delta(\xi)
=S+\frac{\mu-\nu}{S}(a-b)
+\frac{2\mu\nu}{S^3}(a-b)^2
+O\big((a+b)^3\big).
\end{align*}
Substituting this expression into \eqref{Def-eigenvalues} proves \eqref{Exp-lambda}. Formula \eqref{Exp-projector} follows from \eqref{Def-projectors}.
In fact, the linear part in the substitution is simplified by
\begin{align*}
\frac{a+b}{2}-\frac{\mu-\nu}{2S}(a-b)
=\frac{\nu a+\mu b}{S}.
\end{align*}
Moreover, $\delta^{-1}=S^{-1}+O(a+b)$ and $\lambda_+=S+O(a+b)$. Since
$\mb{A}(\xi)=\mb{A}(0)+\operatorname{diag}(a,b)$, inserting these relations into \eqref{Def-projectors} gives
$\mb{P}_-(\xi)=\mb{P}_0+O(a+b)$.
If, for example, $\sigma_1<\sigma_2$, then
\begin{align*}
\lambda_-(\xi)
=\frac{\nu}{S}|\xi|^{2\sigma_1}
+O\big(|\xi|^{2\min\{\sigma_2,2\sigma_1\}}\big).
\end{align*}
This is \eqref{Exp-effective} with $d_*=\frac{\nu}{S}$. The case $\sigma_2<\sigma_1$ is analogous. Finally, differentiating the explicit expressions away from $\xi=0$ and using
\begin{align*}
\left|\partial_\xi^\alpha|\xi|^\beta\right|
\lesssim|\xi|^{\beta-|\alpha|}
\end{align*}
gives \eqref{Est-symbol-low}.
To make the order of the differentiated remainder explicit, when $\sigma_1<\sigma_2$ the analyticity used above yields locally
\begin{align*}
\lambda_-(\xi)-\frac{\nu}{S}|\xi|^{2\sigma_1}
=\frac{\mu}{S}b-\frac{\mu\nu}{S^3}(a-b)^2
+\sum_{k\geqslant3}c_k(a-b)^k.
\end{align*}
Every monomial on the right-hand side has order at least
$2\min\{\sigma_2,2\sigma_1\}=2\rho_*$. Differentiation therefore gives the first estimate in \eqref{Est-symbol-low}. Each entry of $$\mb{P}_--\mb{P}_0$$ is likewise analytic in $(a,b)$, vanishes at $(0,0)$, and consequently starts with terms of order at least $2\sigma_*$. This proves the second estimate. The case $\sigma_2<\sigma_1$ follows in the same way.
\end{proof}

We also recall the fractional heat kernel estimate from
\cite{Narazaki-Reissig=2013} and \cite[Chapter 24.2]{Ebert-Reissig=2018-book}.
\begin{lemma}\label{Lemma-Lp-sharp}
Let $1\leqslant m\leqslant+\infty$, $\sigma>0$, and $s\geqslant0$. Then,
\begin{align}\label{Est-fractional-kernel}
\left\|\ml{F}^{-1}_{\xi\to x}
\left(|\xi|^s\mathrm{e}^{-c_0t|\xi|^{2\sigma}}\right)\right\|_{L^m}
\lesssim
t^{-\frac{n}{2\sigma}(1-\frac{1}{m})-\frac{s}{2\sigma}}
\end{align}
for any $t>0$ and $c_0>0$.
\end{lemma}

The next lemma converts the spectral expansions into kernel estimates. For a matrix-valued kernel $\mb{M}=(M_{jk})_{j,k=1}^2$, we use
\begin{align*}
\|\mb{M}\|_{L^m}:=\sum_{j,k=1}^2\|M_{jk}\|_{L^m}.
\end{align*}
\begin{lemma}\label{Lemma-matrix-kernel}
Let $1\leqslant m\leqslant+\infty$ and $t\geqslant1$. Then,
\begin{align}\label{Est-matrix-kernel}
\left\|\ml{F}^{-1}_{\xi\to x}
\left(\mathrm{e}^{-t\mb{A}(\xi)}\right)\right\|_{L^m}
\lesssim t^{-\frac{n}{2\sigma_*}(1-\frac{1}{m})}.
\end{align}
Moreover,
\begin{align}\label{Est-effective-remainder}
\left\|\ml{F}^{-1}_{\xi\to x}
\left(\mathrm{e}^{-t\mb{A}(\xi)}
-\mathrm{e}^{-d_*t|\xi|^{2\sigma_*}}\mb{P}_0\right)\right\|_{L^m}&\lesssim
\begin{cases}
\mathrm{e}^{-St}\,
t^{-\frac{n}{2\sigma}(1-\frac{1}{m})},
&\sigma_1=\sigma_2=\sigma,\\[1mm]
t^{-\frac{n}{2\sigma_*}(1-\frac{1}{m})-\vartheta_*},
&\sigma_1\neq\sigma_2.
\end{cases}
\end{align}
\end{lemma}
\begin{proof}
If $\sigma_1=\sigma_2=\sigma$, we have
\begin{align*}
\lambda_-(\xi)=|\xi|^{2\sigma},\qquad
\lambda_+(\xi)=|\xi|^{2\sigma}+S,\qquad
\mb{P}_-(\xi)=\mb{P}_0.
\end{align*}
Thus, \eqref{Est-matrix-kernel} and \eqref{Est-effective-remainder} follow directly from Lemma \ref{Lemma-Lp-sharp}.

Let $\sigma_1\neq\sigma_2$, and choose $\chi\in\ml{C}_0^\infty(\mb{R}^n)$ such that $\chi(\xi)=1$ for $|\xi|\leqslant1$ and $\chi(\xi)=0$ for $|\xi|\geqslant2$. In the low-frequency zone, we decompose
\begin{align}\label{Dec-slow-mode}
\mathrm{e}^{-t\lambda_-(\xi)}\mb{P}_-(\xi)
-\mathrm{e}^{-d_*t|\xi|^{2\sigma_*}}\mb{P}_0=
\mathrm{e}^{-t\lambda_-(\xi)}
\big(\mb{P}_-(\xi)-\mb{P}_0\big)
+\big(\mathrm{e}^{-t\lambda_-(\xi)}
-\mathrm{e}^{-d_*t|\xi|^{2\sigma_*}}\big)\mb{P}_0.
\end{align}
We use the following standard consequence of the radial Fourier multiplier estimate underlying Lemma \ref{Lemma-Lp-sharp}. Suppose that a multiplier $M_t$ is supported in a fixed ball and, for all multi-indices up to a sufficiently large order, satisfies
\begin{align}\label{Est-scaled-multiplier}
|\partial_\xi^\alpha M_t(\xi)|
\lesssim t^\kappa|\xi|^{\beta-|\alpha|}
\mathrm{e}^{-ct|\xi|^{2\sigma_*}}.
\end{align}
Then the change of variables $\eta=t^{\frac{1}{2\sigma_*}}\xi$, followed by the low-frequency kernel estimate, gives
\begin{align}\label{Est-scaled-kernel}
\left\|\ml{F}^{-1}_{\xi\to x}(M_t)\right\|_{L^m}
\lesssim
t^{-\frac{n}{2\sigma_*}(1-\frac{1}{m})
-\frac{\beta}{2\sigma_*}+\kappa}.
\end{align}
The constants are uniform for $t\geqslant1$; derivatives falling on the fixed cut-off are supported away from the origin and are exponentially small.
The first term on the right-hand side contains the factor
$|\xi|^{2\sigma_*}\mathrm{e}^{-ct|\xi|^{2\sigma_*}}$ by \eqref{Est-symbol-low}, and therefore gains $t^{-1}$. For the second term, the fundamental theorem of calculus gives
\begin{align*}
\mathrm{e}^{-t\lambda_-(\xi)}
-\mathrm{e}^{-d_*t|\xi|^{2\sigma_*}}=-t\big(\lambda_-(\xi)-d_*|\xi|^{2\sigma_*}\big)
\int_0^1
\mathrm{e}^{-t[\theta\lambda_-(\xi)+(1-\theta)d_*|\xi|^{2\sigma_*}]}
\,\mathrm{d}\theta.
\end{align*}
By \eqref{Est-eigenvalues} and \eqref{Est-symbol-low}, this term is bounded, together with its derivatives, by a finite sum of multipliers having the form
\begin{align*}
t|\xi|^{2\rho_*-|\alpha|}
\mathrm{e}^{-ct|\xi|^{2\sigma_*}}.
\end{align*}
After the change of variables $\eta=t^{\frac{1}{2\sigma_*}}\xi$, this gives the factor
\begin{align*}
t^{1-\frac{\rho_*}{\sigma_*}}=t^{-\vartheta_*}.
\end{align*}
More explicitly, estimate \eqref{Est-scaled-kernel} is applied with
$(\kappa,\beta)=(0,2\sigma_*)$ to the projector term and with
$(\kappa,\beta)=(1,2\rho_*)$ to the eigenvalue term. These choices give the additional factors $t^{-1}$ and
$t^{1-\frac{\rho_*}{\sigma_*}}=t^{-\vartheta_*}$, respectively.
The usual Fourier kernel estimate, applied after this scaling, yields
\begin{align*}
\left\|\ml{F}^{-1}\left(
\chi\big(\mathrm{e}^{-t\lambda_-}\mb{P}_-
-\mathrm{e}^{-d_*t|\xi|^{2\sigma_*}}\mb{P}_0\big)\right)\right\|_{L^m}
\lesssim
t^{-\frac{n}{2\sigma_*}(1-\frac{1}{m})-\vartheta_*}.
\end{align*}
Here the projector contribution, which gains $t^{-1}$, is no larger because $0<\vartheta_*\leqslant1$.

For $|\xi|\geqslant1$, the explicit formulas \eqref{Def-eigenvalues} and \eqref{Def-projectors} imply
\begin{align*}
\lambda_-(\xi)\gtrsim|\xi|^{2\sigma_*},\qquad
\lambda_+(\xi)\gtrsim1+|\xi|^{2\sigma^*},\qquad
|\partial_\xi^\alpha\mb{P}_{\pm}(\xi)|
\lesssim|\xi|^{-|\alpha|}.
\end{align*}
Consequently, the inverse Fourier transforms of all high-frequency terms are $O(\mathrm{e}^{-ct})$ in $L^m$. The fast mode in \eqref{Eq-spectral-decomposition} has the same bound because $\lambda_+(\xi)\geqslant c>0$. Combining the low- and high-frequency estimates proves \eqref{Est-effective-remainder}. Estimate \eqref{Est-matrix-kernel} follows by adding the effective heat kernel and using Lemma \ref{Lemma-Lp-sharp}.
Here are the details behind the exponentially small bounds. On the support of $\chi$, one has $\lambda_+\geqslant S$ by the min--max principle applied to the symmetric matrix above. The low-frequency expansions
\begin{align*}
\lambda_+=S+O(a+b)\ \ \mbox{and}\ \ \mb{P}_+=\mb{I}-\mb{P}_0+O(a+b)
\end{align*}
 allow us to use the same radial multiplier estimate as in \eqref{Est-scaled-kernel}. Differentiation of
$\mathrm{e}^{-t\lambda_+}\mb{P}_+$ produces only finitely many polynomial factors in $t$, and they are absorbed by
$$t^N\mathrm{e}^{-St}\lesssim\mathrm{e}^{-c't}.$$
For the high-frequency terms, the displayed symbol bounds and the same argument give exponentially small weighted $L^2$ norms of the multipliers and their derivatives. If $N>\frac{n}{2}$ is an integer and $K=\ml{F}^{-1}(M)$, then weighted Cauchy--Schwarz and Plancherel's identity yield
\begin{align*}
\|K\|_{L^1}
\lesssim\sum_{|\alpha|\leqslant N}
\|\partial_\xi^\alpha M\|_{L^2},
\qquad
\|K\|_{L^\infty}\lesssim\|M\|_{L^1}.
\end{align*}
Thus, both endpoint kernel norms are $O(\mathrm{e}^{-c't})$, and interpolation proves the claim for every $1\leqslant m\leqslant+\infty$.
\end{proof}

\subsection{Refined asymptotic behavior of solutions for large time}
\hspace{5mm}For $\bar{w}\in\{\bar{u},\bar{v}\}$, set
\begin{align}\label{Def-linear-profile}
\bar{w}^{\mathrm{prof}}(t,x)
:=\frac{\gamma}{\mu+\nu}\,
\ml{G}_*(t,|D|)\big(\bar{u}_0(x)+\bar{v}_0(x)\big),
\end{align}
where $\gamma=\nu$ if $\bar{w}=\bar{u}$, and $\gamma=\mu$ if $\bar{w}=\bar{v}$.
\begin{prop}\label{Prop-Linear-Lp}
Let $\bar{u}_0,\bar{v}_0\in L^m\cap L^r$ with
$1\leqslant r\leqslant m\leqslant+\infty$. Then,
\begin{align}\label{Est-linear-LrLm}
\|\bar{w}(t,\cdot)\|_{L^m}
\lesssim
(1+t)^{-\frac{n}{2\sigma_*}(\frac{1}{r}-\frac{1}{m})}
\|(\bar{u}_0,\bar{v}_0)\|_{(L^m\cap L^r)^2}.
\end{align}
If $\sigma_1=\sigma_2=\sigma$, then
\begin{align}\label{Est-linear-remainder-equal}
\|\bar{w}(t,\cdot)-\bar{w}^{\mathrm{prof}}(t,\cdot)\|_{L^m}
\lesssim
\mathrm{e}^{-(\mu+\nu)t}
(1+t)^{-\frac{n}{2\sigma}(\frac{1}{r}-\frac{1}{m})}
\|\mu\bar{u}_0-\nu\bar{v}_0\|_{L^m\cap L^r}.
\end{align}
If $\sigma_1\neq\sigma_2$, then
\begin{align}\label{Est-linear-remainder-different}
\|\bar{w}(t,\cdot)-\bar{w}^{\mathrm{prof}}(t,\cdot)\|_{L^m}
\lesssim
(1+t)^{-\frac{n}{2\sigma_*}(\frac{1}{r}-\frac{1}{m})-\vartheta_*}
\|(\bar{u}_0,\bar{v}_0)\|_{(L^m\cap L^r)^2}.
\end{align}
\end{prop}
\begin{proof}
For $t\geqslant1$, Young's convolution inequality and Lemma \ref{Lemma-matrix-kernel} give \eqref{Est-linear-LrLm} as well as \eqref{Est-linear-remainder-different}.
Indeed, choose $\ell\in[1,+\infty]$ from
\begin{align*}
1+\frac{1}{m}=\frac{1}{\ell}+\frac{1}{r}.
\end{align*}
Then
\begin{align*}
\frac{n}{2\sigma_*}\left(1-\frac{1}{\ell}\right)
=\frac{n}{2\sigma_*}\left(\frac{1}{r}-\frac{1}{m}\right),
\end{align*}
so Young's inequality with the $L^\ell$ kernel bounds in Lemma \ref{Lemma-matrix-kernel} gives exactly the time powers in
\eqref{Est-linear-LrLm} and \eqref{Est-linear-remainder-different}.
In the equal-order case, the exact identity
\begin{align*}
\begin{pmatrix}\bar{u}\\ \bar{v}\end{pmatrix}
-\mb{P}_0\ml{G}_*(t,|D|)
\begin{pmatrix}\bar{u}_0\\ \bar{v}_0\end{pmatrix}
=\frac{\mathrm{e}^{-St}}{S}\ml{G}_*(t,|D|)
\begin{pmatrix}
\mu\bar{u}_0-\nu\bar{v}_0\\
-\mu\bar{u}_0+\nu\bar{v}_0
\end{pmatrix}
\end{align*}
proves \eqref{Est-linear-remainder-equal}.

It remains to treat $0\leqslant t\leqslant1$. Let
$$\mb{H}(t,|D|):=\operatorname{diag}(\ml{G}_{\sigma_1}(t,|D|),\ml{G}_{\sigma_2}(t,|D|)),$$ where the diagonal entries have symbols
$\mathrm{e}^{-t|\xi|^{2\sigma_1}}$ and
$\mathrm{e}^{-t|\xi|^{2\sigma_2}}$, respectively. The variation-of-constants formula reads
\begin{align*}
\mb{S}(t,|D|)
=\mb{H}(t,|D|)
-\int_0^t\mb{H}(t-\tau,|D|)\mb{B}\mb{S}(\tau,|D|)\,\mathrm{d}\tau,
\end{align*}
where
\begin{align*}
\mb{B}:=
\begin{pmatrix}\mu&-\nu\\-\mu&\nu\end{pmatrix}.
\end{align*}
The scalar fractional heat kernels have uniformly bounded $L^1$ norms. Hence, Gronwall's inequality gives
\begin{align*}
\sup_{0\leqslant t\leqslant1}
\|\mb{S}(t,|D|)\|_{(L^m)^2\to(L^m)^2}\lesssim1
\end{align*}
for every $1\leqslant m\leqslant+\infty$. The profile operator is bounded on the same interval. This completes all estimates with the factors $1+t$.
\end{proof}

\begin{remark}
For $\sigma_1\neq\sigma_2$, the additional factor in \eqref{Est-linear-remainder-different} is
\begin{align*}
(1+t)^{-\vartheta_*}
=
\begin{cases}
(1+t)^{-\frac{\sigma^*-\sigma_*}{\sigma_*}},
&\sigma_*<\sigma^*<2\sigma_*,\\[1mm]
(1+t)^{-1},
&\sigma^*\geqslant2\sigma_*.
\end{cases}
\end{align*}
Thus, $\sigma^*=2\sigma_*$ is the threshold at which the correction
$|\xi|^{2\sigma^*}$ and the quadratic correction
$|\xi|^{4\sigma_*}$ in \eqref{Exp-lambda} have the same order.
\end{remark}

We next pass from the profile in \eqref{Def-linear-profile} to its mass profile.
\begin{prop}\label{Prop-Lm-profile}
Let $\bar{u}_0,\bar{v}_0\in L^1$. Then,
\begin{align}\label{Est-linear-mass-profile}
\lim\limits_{t\to+\infty}
t^{\frac{n}{2\sigma_*}(1-\frac{1}{m})}
\left\|\bar{w}(t,\cdot)
-\frac{\gamma}{\mu+\nu}G_*(t,\cdot)
P_{\bar{u}_0+\bar{v}_0}\right\|_{L^m}=0
\end{align}
for every $1\leqslant m\leqslant+\infty$.
\end{prop}
\begin{proof}
Let
\begin{align*}
\bar{g}_0(x):=\frac{\gamma}{\mu+\nu}
\big(\bar{u}_0(x)+\bar{v}_0(x)\big).
\end{align*}
Using
\begin{align*}
G_*(t,x-y)-G_*(t,x)
=-\int_0^1y\cdot\nabla G_*(t,x-\theta y)\,\mathrm{d}\theta,
\end{align*}
we split at $|y|=t^{\frac{1}{4\sigma_*}}$ and obtain
\begin{align*}
\|\ml{G}_*(t,|D|)\bar{g}_0-G_*(t,\cdot)P_{\bar{g}_0}\|_{L^m}&\lesssim
t^{\frac{1}{4\sigma_*}}\|\nabla G_*(t,\cdot)\|_{L^m}
\|\bar{g}_0\|_{L^1}
+\|G_*(t,\cdot)\|_{L^m}
\|\bar{g}_0\|_{L^1(|x|\geqslant t^{\frac{1}{4\sigma_*}})}\\
&=o\left(t^{-\frac{n}{2\sigma_*}(1-\frac{1}{m})}\right).
\end{align*}
More precisely, writing $R=t^{\frac{1}{4\sigma_*}}$ and
$a_m=\frac{n}{2\sigma_*}(1-\frac{1}{m})$, Lemma \ref{Lemma-Lp-sharp} gives
\begin{align*}
t^{a_m}R\|\nabla G_*(t,\cdot)\|_{L^m}
\lesssim t^{-\frac{1}{4\sigma_*}}\longrightarrow0,
\end{align*}
whereas
\begin{align*}
t^{a_m}\|G_*(t,\cdot)\|_{L^m}
\|\bar{g}_0\|_{L^1(|x|\geqslant R)}
\lesssim
\|\bar{g}_0\|_{L^1(|x|\geqslant R)}
\longrightarrow0.
\end{align*}
Here Lemma \ref{Lemma-Lp-sharp} was used with $\sigma=\sigma_*$, while the last term is small by the integrability of $\bar{g}_0$. Combining this estimate with Lemma \ref{Lemma-matrix-kernel} and Young's convolution inequality, applied directly to the $L^1$ initial data, proves \eqref{Est-linear-mass-profile}.
Indeed, after multiplication by $t^{a_m}$, the matrix remainder is bounded by
$Ct^{-\vartheta_*}\|(\bar{u}_0,\bar{v}_0)\|_{(L^1)^2}$ if
$\sigma_1\neq\sigma_2$, and by
$C\mathrm{e}^{-St}\|(\bar{u}_0,\bar{v}_0)\|_{(L^1)^2}$ if
$\sigma_1=\sigma_2$. Both quantities tend to zero.
\end{proof}

\begin{prop}\label{Prop-Linear-Optimal-L2}
Let $\bar{u}_0,\bar{v}_0\in L^1$ and
$P_{\bar{u}_0+\bar{v}_0}\neq0$. Then,
\begin{align*}
t^{-\frac{n}{4\sigma_*}}
|P_{\bar{u}_0+\bar{v}_0}|
\lesssim\|\bar{w}(t,\cdot)\|_{L^2}
\lesssim
t^{-\frac{n}{4\sigma_*}}
\|(\bar{u}_0,\bar{v}_0)\|_{(L^1)^2}
\end{align*}
as large time $t\gg1$.
\end{prop}
\begin{proof}
The upper bound follows from Lemma \ref{Lemma-matrix-kernel} and Young's convolution inequality. By Plancherel's identity and polar coordinates,
\begin{align}\label{Est-effective-L2}
\|G_*(t,\cdot)\|_{L^2}^2
\approx\int_{\mb{R}^n}\mathrm{e}^{-2d_*t|\xi|^{2\sigma_*}}\,\mathrm{d}\xi
\approx t^{-\frac{n}{2\sigma_*}}.
\end{align}
On the other hand, Proposition \ref{Prop-Lm-profile} with $m=2$ gives
\begin{align*}
\|\bar{w}(t,\cdot)\|_{L^2}
&\geqslant
\frac{\gamma}{\mu+\nu}\|G_*(t,\cdot)\|_{L^2}
|P_{\bar{u}_0+\bar{v}_0}|-
\left\|\bar{w}(t,\cdot)
-\frac{\gamma}{\mu+\nu}G_*(t,\cdot)
P_{\bar{u}_0+\bar{v}_0}\right\|_{L^2}\\
&\gtrsim
t^{-\frac{n}{4\sigma_*}}
|P_{\bar{u}_0+\bar{v}_0}|
-o\left(t^{-\frac{n}{4\sigma_*}}\right).
\end{align*}
This proves the lower bound.
\end{proof}
\section{Semilinear nonlocal heat exchanger system}\label{Section_Existence}\setcounter{equation}{0}
\subsection{Philosophy of our proofs}\label{Philosophy-Global}
\hspace{5mm}Before proving the local/global in-time existence results and the lower bound estimates for the lifespan $T_{\varepsilon}$, as preparations, we at first explain our philosophy.

For any $T\in(0,+\infty]$, we introduce the evolution space for the vector-valued solution $\ml{U}=\ml{U}(t,x):=(u,v)^{\mathrm{T}}$ by
\begin{align*}
X_T:=\big(&\ml{C}([0,T],L^1)\cap L^{\infty}([0,T],L^{\infty})\big)^2
\end{align*}
equipping the time-weighted norm
\begin{align*}
\|\ml{U}\|_{X_T}:=
\sup\limits_{t\in[0,T]}
\sum\limits_{w\in\{u,v\}}
\left(\|w(t,\cdot)\|_{L^1}
+(1+t)^{\frac{n}{2\sigma_*}}\|w(t,\cdot)\|_{L^{\infty}}\right).
\end{align*}
With this norm, $X_T$ is a Banach space. Let $\mb{S}(t,|D|)$ be the matrix-valued solution operator whose Fourier symbol is $\mathrm{e}^{-t\mb{A}(\xi)}$. From Duhamel's principle, we construct
\begin{align}\label{Def-nonlinear-map}
\ml{N}[\ml{U}](t)
:=\varepsilon\mb{S}(t,|D|)\ml{U}_0
+\int_0^t\mb{S}(t-\tau,|D|)
\begin{pmatrix}|u(\tau)|^p\\ |v(\tau)|^q\end{pmatrix}\,\mathrm{d}\tau,
\qquad
\ml{U}_0:=\begin{pmatrix}u_0\\v_0\end{pmatrix}.
\end{align}
Its two components agree with the right-hand side of \eqref{Equa-01}.

The assumption $(u_0,v_0)\in(L^{\infty}\cap L^1)^2$ and Proposition \ref{Prop-Linear-Lp} imply
\begin{align}\label{Est-linear-XT}
\|\mb{S}(t,|D|)\ml{U}_0\|_{X_T}
\leqslant C_0\|(u_0,v_0)\|_{(L^{\infty}\cap L^1)^2}
\end{align}
uniformly for $T\in(0,+\infty]$. We seek a fixed point in
\begin{align*}
\ml{B}_{\kappa}(X_T)
:=\left\{\ml{U}\in X_T:\ \|\ml{U}\|_{X_T}\leqslant\kappa\right\},
\qquad
\kappa:=2\varepsilon C_0
\|(u_0,v_0)\|_{(L^{\infty}\cap L^1)^2}.
\end{align*}
For $r>1$, put
\begin{align}\label{Def-Cr}
C_r(T):=\int_0^T(1+\tau)^{-\frac{n}{2\sigma_*}(r-1)}\,\mathrm{d}\tau.
\end{align}
The estimates proved in the next subsection give
\begin{align}
\|\ml{N}[\ml{U}]\|_{X_T}
&\leqslant\varepsilon C_0\|(u_0,v_0)\|_{(L^{\infty}\cap L^1)^2}
+C\sum\limits_{r\in\{p,q\}}C_r(T)\|\ml{U}\|_{X_T}^{r},
\label{Est-Important-01}\\
\|\ml{N}[\ml{U}]-\ml{N}[\ml{V}]\|_{X_T}
&\leqslant C\|\ml{U}-\ml{V}\|_{X_T}
\sum\limits_{r\in\{p,q\}}C_r(T)
\left(\|\ml{U}\|_{X_T}^{r-1}+\|\ml{V}\|_{X_T}^{r-1}\right).
\label{Est-Important-02}
\end{align}
Since $C_r(T)\to0$ as $T\to0^+$, these inequalities prove local in-time existence and uniqueness for every fixed $\varepsilon>0$.
Indeed, for the fixed radius $\kappa$ associated with an arbitrary
$\varepsilon>0$, one first chooses $T>0$ so small that
\begin{align*}
C\sum_{r\in\{p,q\}}C_r(T)\kappa^{r-1}
\leqslant\frac{1}{4}.
\end{align*}
Then \eqref{Est-Important-01} and \eqref{Est-Important-02} show, after enlarging the harmless constant if necessary, that $\ml{N}$ is a self-map and a contraction on $\ml{B}_\kappa(X_T)$.
For the small-data estimates, choose $\varepsilon$ such that $\kappa\leqslant1$ and set
\begin{align*}
r_0:=\min\{p,q\}.
\end{align*}
Because $C_r(T)\leqslant C_{r_0}(T)$ and $\kappa^{r-1}\leqslant\kappa^{r_0-1}$ for $r\in\{p,q\}$, it is sufficient that
\begin{align}\label{Cond-contraction}
C_*C_{r_0}(T)\kappa^{r_0-1}\leqslant\frac{1}{4}
\end{align}
for a suitable constant $C_*>0$. Under \eqref{Cond-contraction}, $\ml{N}$ maps $\ml{B}_{\kappa}(X_T)$ into itself and has Lipschitz constant at most $\frac12$ there.
More explicitly, $C_*$ may be fixed large enough that
\begin{align*}
\|\ml{N}[\ml{U}]\|_{X_T}
&\leqslant\frac{\kappa}{2}
+C_*C_{r_0}(T)\kappa^{r_0}
\leqslant\frac{3\kappa}{4},\\
\|\ml{N}[\ml{U}]-\ml{N}[\ml{V}]\|_{X_T}
&\leqslant
2C_*C_{r_0}(T)\kappa^{r_0-1}
\|\ml{U}-\ml{V}\|_{X_T}
\leqslant\frac{1}{2}\|\ml{U}-\ml{V}\|_{X_T}.
\end{align*}
Thus, the stated condition supplies both parts of the Banach fixed point argument.

We next separate the discussion according to the value of $r_0$, motivated by the recent work \cite{Chen-Girardi=2025}.
\begin{description}
\item[{Super-critical} Case: $r_0>p_{\mathrm{Fuj}}(\frac{n}{\sigma_*})$.]
In this case, $C_{r_0}(+\infty)<+\infty$. Therefore, \eqref{Cond-contraction} holds with $T=+\infty$ after choosing $\varepsilon_0>0$ sufficiently small.
Indeed, since
$\kappa=2C_0\varepsilon\|(u_0,v_0)\|_{(L^\infty\cap L^1)^2}$,
it is enough to choose $\varepsilon_0$ so that
\begin{align*}
C_*C_{r_0}(+\infty)
\left(2C_0\varepsilon_0
\|(u_0,v_0)\|_{(L^\infty\cap L^1)^2}\right)^{r_0-1}
\leqslant\frac{1}{4}.
\end{align*}
The fixed point satisfies
\begin{align}\label{Est-global-X}
\sup\limits_{t\in[0,+\infty)}
\sum\limits_{w\in\{u,v\}}
\left(\|w(t,\cdot)\|_{L^1}
+(1+t)^{\frac{n}{2\sigma_*}}\|w(t,\cdot)\|_{L^{\infty}}\right)
\lesssim\varepsilon\|(u_0,v_0)\|_{(L^{\infty}\cap L^1)^2}.
\end{align}
Interpolation between $L^1$ and $L^{\infty}$ yields \eqref{Est-03}.

\item[{Sub-critical} Case: $r_0<p_{\mathrm{Fuj}}(\frac{n}{\sigma_*})$.]
Putting
\begin{align*}
\beta:=1-\frac{n}{2\sigma_*}(r_0-1)>0,
\end{align*}
we have $C_{r_0}(T)\lesssim(1+T)^{\beta}$. For all sufficiently small $\varepsilon$, condition \eqref{Cond-contraction} holds whenever
\begin{align}\label{Est-02}
0<T\leqslant c\varepsilon^{-\frac{r_0-1}{\beta}},
\end{align}
where $c>0$ is independent of $\varepsilon$.
To verify \eqref{Cond-contraction}, observe that for the time in
\eqref{Est-02},
\begin{align*}
C_{r_0}(T)\kappa^{r_0-1}
&\lesssim
(1+T)^\beta\varepsilon^{r_0-1}
\lesssim c^\beta
\end{align*}
for all sufficiently small $\varepsilon$, where the fixed norm of the initial data and the constants in $\kappa$ have been absorbed into the implicit constant. Choosing $c$ small proves the claim.
This proves the asserted lower bound for $T_{\varepsilon}$.

\item[{Critical} Case: $r_0=p_{\mathrm{Fuj}}(\frac{n}{\sigma_*})$.]
Now $C_{r_0}(T)=\ln(1+T)$. Choosing $c>0$ sufficiently small, condition \eqref{Cond-contraction} is valid for
\begin{align*}
0<T\leqslant\exp\left(c\varepsilon^{-r_0+1}\right)
=\exp\left(c\varepsilon^{-\frac{2\sigma_*}{n}}\right).
\end{align*}
Indeed,
\begin{align*}
C_{r_0}(T)\kappa^{r_0-1}
=\ln(1+T)\kappa^{r_0-1}\lesssim c
\end{align*}
for the displayed choice of $T$ and all sufficiently small
$\varepsilon$. Taking $c$ small gives \eqref{Cond-contraction}.
This proves the critical lower bound for $T_{\varepsilon}$.
\end{description}

\subsection{Proofs of Theorem \ref{Thm=Global-Solution} and Theorem \ref{Thm=lower-lifespan}}
\hspace{5mm}We prove the estimates used above. The definition of $X_T$ and interpolation between the $L^1$ and $L^{\infty}$ norms give, for $r>1$ and $1\leqslant m\leqslant+\infty$,
\begin{align}\label{Est-power-Lm}
\||w(\tau,\cdot)|^r\|_{L^m}
\lesssim
(1+\tau)^{-\frac{n}{2\sigma_*}(r-\frac{1}{m})}
\|\ml{U}\|_{X_T}^{r},
\qquad w\in\{u,v\}.
\end{align}
Indeed, interpolation gives
\begin{align*}
\|w(\tau)\|_{L^{mr}}
&\leqslant
\|w(\tau)\|_{L^1}^{\frac{1}{mr}}
\|w(\tau)\|_{L^\infty}^{1-\frac{1}{mr}}\\
&\lesssim
(1+\tau)^{-\frac{n}{2\sigma_*}
\left(1-\frac{1}{mr}\right)}
\|\ml{U}\|_{X_T}.
\end{align*}
Raising this inequality to the power $r$ proves
\eqref{Est-power-Lm}; the endpoint $m=+\infty$ follows directly from the $L^\infty$ part of the $X_T$ norm.
Using the bounded $L^1-L^1$ estimate from Proposition \ref{Prop-Linear-Lp}, we obtain
\begin{align*}
\|\ml{U}_{\nlin}(t,\cdot)\|_{(L^1)^2}
&\lesssim\sum\limits_{r\in\{p,q\}}
\int_0^t(1+\tau)^{-\frac{n}{2\sigma_*}(r-1)}\,\mathrm{d}\tau
\|\ml{U}\|_{X_T}^{r}.
\end{align*}
For the $L^{\infty}$ estimate, we split the time integral at $\frac{t}{2}$. The $(L^{\infty}\cap L^1)-L^{\infty}$ estimate is used on $[0,\frac{t}{2}]$, whereas the bounded $L^{\infty}-L^{\infty}$ estimate is used on $[\frac{t}{2},t]$.
Before multiplying by the time weight, these two semigroup estimates give
\begin{align*}
\|\ml{U}_{\nlin}(t)\|_{(L^\infty)^2}
\lesssim\sum_{r\in\{p,q\}}\Bigg(
&\int_0^{\frac{t}{2}}
(1+t-\tau)^{-\frac{n}{2\sigma_*}}
\||w(\tau)|^r\|_{L^1\cap L^\infty}\,\mathrm{d}\tau+\int_{\frac{t}{2}}^t
\||w(\tau)|^r\|_{L^\infty}\,\mathrm{d}\tau
\Bigg),
\end{align*}
where $w=u$ when $r=p$ and $w=v$ when $r=q$. Substitution of
\eqref{Est-power-Lm} with $m=1$ and $m=+\infty$ gives the next display.
\begin{align*}
&(1+t)^{\frac{n}{2\sigma_*}}
\|\ml{U}_{\nlin}(t,\cdot)\|_{(L^{\infty})^2}\\
&\quad\lesssim
\sum\limits_{r\in\{p,q\}}
\left(
\int_0^{\frac{t}{2}}(1+\tau)^{-\frac{n}{2\sigma_*}(r-1)}\,\mathrm{d}\tau
+(1+t)^{\frac{n}{2\sigma_*}}
\int_{\frac{t}{2}}^t(1+\tau)^{-\frac{nr}{2\sigma_*}}\,\mathrm{d}\tau
\right)\|\ml{U}\|_{X_T}^{r}\\
&\quad\lesssim
\sum\limits_{r\in\{p,q\}}C_r(T)\|\ml{U}\|_{X_T}^{r}.
\end{align*}
Here we used $1+t-\tau\approx1+t$ for $\tau\in[0,\frac{t}{2}]$ and $1+\tau\approx1+t$ for $\tau\in[\frac{t}{2},t]$. This proves \eqref{Est-Important-01}.

For $r>1$, the pointwise inequality
\begin{align}\label{Est-power-difference}
\big||a|^r-|b|^r\big|
\lesssim|a-b|\big(|a|^{r-1}+|b|^{r-1}\big)
\end{align}
More precisely, if $w$ and $z$ are the corresponding components of
$\ml{U}$ and $\ml{V}$, respectively, then
\begin{align*}
\big\||w(\tau)|^r-|z(\tau)|^r\big\|_{L^1}
&\lesssim
(1+\tau)^{-\frac{n}{2\sigma_*}(r-1)}
\|\ml{U}-\ml{V}\|_{X_T}
\big(\|\ml{U}\|_{X_T}^{r-1}
+\|\ml{V}\|_{X_T}^{r-1}\big),\\
\big\||w(\tau)|^r-|z(\tau)|^r\big\|_{L^\infty}
&\lesssim
(1+\tau)^{-\frac{nr}{2\sigma_*}}
\|\ml{U}-\ml{V}\|_{X_T}
\big(\|\ml{U}\|_{X_T}^{r-1}
+\|\ml{V}\|_{X_T}^{r-1}\big).
\end{align*}
The first bound follows by taking the $L^1$ norm of
\eqref{Est-power-difference} and placing the factor of order $r-1$ in
$L^\infty$; the second follows by placing both factors in
$L^\infty$. Repeating the preceding time split proves
\eqref{Est-Important-02}.
Hence, the Banach contraction principle completes the existence, uniqueness, and lifespan arguments. The strong continuity of the scalar fractional heat semigroups on $L^m$ for $1\leqslant m<+\infty$, the bounded perturbation by the exchange matrix, and interpolation imply
\begin{align*}
(u,v)\in\ml{C}([0,T],L^{m_1})\times\ml{C}([0,T],L^{m_2})
\end{align*}
for $1\leqslant m_1,m_2<+\infty$. The smoothing estimates for positive time give continuity into $L^{\infty}$ on $(0,T]$.
Indeed, \eqref{Est-power-Lm} implies
$$\ml{F}\in L^1_{\mathrm{loc}}([0,T];(L^1\cap L^\infty)^2).$$
The strong continuity of $\mb{S}(t,|D|)$ on every finite $L^m$, together with a standard dominated convergence argument in the Duhamel integral, yields the asserted $L^m$ time continuity. On each interval $[\delta,T]$ with $\delta>0$, the fractional smoothing estimates give the same conclusion in $L^\infty$; since $\delta$ is arbitrary, the $L^\infty$ continuity holds on $(0,T]$.
This proves Theorem \ref{Thm=Global-Solution} and Theorem \ref{Thm=lower-lifespan}.

\subsection{Proof of Theorem \ref{Thm=Large-time}}
\hspace{5mm}Let
\begin{align*}
\ml{U}:=\begin{pmatrix}u\\v\end{pmatrix},
\qquad
\ml{F}(t,x):=\begin{pmatrix}|u(t,x)|^p\\|v(t,x)|^q\end{pmatrix},
\qquad
\ml{D}(t):=\int_0^t\mb{S}(t-\tau,|D|)\ml{F}(\tau)\,\mathrm{d}\tau.
\end{align*}
The decay estimate \eqref{Est-03} shows
\begin{align}\label{Est-F-integrable}
\|\ml{F}(\tau,\cdot)\|_{(L^1)^2}
\lesssim
\sum\limits_{r\in\{p,q\}}\varepsilon^r
(1+\tau)^{-\frac{n}{2\sigma_*}(r-1)}
\|(u_0,v_0)\|_{(L^{\infty}\cap L^1)^2}^{r}.
\end{align}
For example,
\begin{align*}
\||u(\tau)|^p\|_{L^1}
&\leqslant
\|u(\tau)\|_{L^\infty}^{p-1}\|u(\tau)\|_{L^1}\\
&\lesssim
\varepsilon^p(1+\tau)^{-\frac{n}{2\sigma_*}(p-1)}
\|(u_0,v_0)\|_{(L^\infty\cap L^1)^2}^{p}.
\end{align*}
The same calculation for $v$ with exponent $q$ proves
\eqref{Est-F-integrable}.
Since $\min\{p,q\}>1+\frac{2\sigma_*}{n}$, it follows that
\begin{align}\label{Est-F-space-time}
\ml{F}\in\big(L^1((0,+\infty)\times\mb{R}^n)\big)^2.
\end{align}
For convenience, define the vector-valued masses
\begin{align*}
\mb{M}_{\ml{F}}(\tau):=\int_{\mb{R}^n}\ml{F}(\tau,y)\,\mathrm{d}y,
\qquad
\mb{M}_{\ml{F}}:=\int_0^{+\infty}\mb{M}_{\ml{F}}(\tau)\,\mathrm{d}\tau.
\end{align*}

Fix $1\leqslant m\leqslant+\infty$ and put
\begin{align*}
a_m:=\frac{n}{2\sigma_*}\left(1-\frac{1}{m}\right).
\end{align*}
Motivated by \cite{Ikehata-Takeda=2017}, we decompose the nonlinear part into five terms:
\begin{align}\label{Dec-nonlinear-profile}
\ml{D}(t)-G_*(t,\cdot)\mb{P}_0\mb{M}_{\ml{F}}
=\sum\limits_{j=1}^5\mb{A}_j(t),
\end{align}
where
\begin{align*}
\mb{A}_1(t)
&:=\int_0^{\frac{t}{2}}
\big(\mb{S}(t-\tau,|D|)-\mb{P}_0\ml{G}_*(t-\tau,|D|)\big)
\ml{F}(\tau)\,\mathrm{d}\tau,\\
\mb{A}_2(t)
&:=\int_0^{\frac{t}{2}}\mb{P}_0
\big(\ml{G}_*(t-\tau,|D|)-\ml{G}_*(t,|D|)\big)
\ml{F}(\tau)\,\mathrm{d}\tau,\\
\mb{A}_3(t,x)
&:=\int_0^{\frac{t}{2}}\mb{P}_0
\left(\ml{G}_*(t,|D|)\ml{F}(\tau,x)
-G_*(t,x)\mb{M}_{\ml{F}}(\tau)\right)\,\mathrm{d}\tau,\\
\mb{A}_4(t)
&:=\int_{\frac{t}{2}}^t\mb{S}(t-\tau,|D|)\ml{F}(\tau)\,\mathrm{d}\tau,\\
\mb{A}_5(t,x)
&:=-G_*(t,x)\mb{P}_0
\int_{\frac{t}{2}}^{+\infty}\mb{M}_{\ml{F}}(\tau)\,\mathrm{d}\tau.
\end{align*}
To verify \eqref{Dec-nonlinear-profile}, notice first that
\begin{align*}
\mb{A}_1+\mb{A}_2+\mb{A}_3
=\int_0^{\frac{t}{2}}\mb{S}(t-\tau,|D|)
\ml{F}(\tau)\,\mathrm{d}\tau
-G_*(t,\cdot)\mb{P}_0
\int_0^{\frac{t}{2}}\mb{M}_{\ml{F}}(\tau)\,\mathrm{d}\tau.
\end{align*}
Adding $\mb{A}_4$ completes the Duhamel integral, while adding
$\mb{A}_5$ replaces the truncated mass integral by the integral over
$[0,+\infty)$. This gives precisely the left-hand side of
\eqref{Dec-nonlinear-profile}.
We show that every term is $o(t^{-a_m})$ in $(L^m)^2$.

For $\mb{A}_1$, Lemma \ref{Lemma-matrix-kernel}, Young's convolution inequality, and \eqref{Est-F-space-time} give
\begin{align*}
t^{a_m}\|\mb{A}_1(t)\|_{(L^m)^2}
\lesssim
\begin{cases}
\mathrm{e}^{-ct}\|\ml{F}\|_{(L^1_{t,x})^2},&\sigma_1=\sigma_2,\\
t^{-\vartheta_*}\|\ml{F}\|_{(L^1_{t,x})^2},&\sigma_1\neq\sigma_2,
\end{cases}
\end{align*}
which converges to zero.
Indeed, in the different-order case, $t-\tau\approx t$ on
$0\leqslant\tau\leqslant\frac{t}{2}$, and hence
\begin{align*}
t^{a_m}\|\mb{A}_1(t)\|_{(L^m)^2}
&\lesssim
t^{a_m}\int_0^{\frac{t}{2}}
(t-\tau)^{-a_m-\vartheta_*}
\|\ml{F}(\tau)\|_{(L^1)^2}\,\mathrm{d}\tau\\
&\lesssim
t^{-\vartheta_*}\|\ml{F}\|_{(L^1_{t,x})^2}.
\end{align*}
In the equal-order case, the same computation contains
$\mathrm{e}^{-S(t-\tau)}$, which is bounded by
$\mathrm{e}^{-\frac{S}{2}t}$ on this interval and gives the asserted exponentially small quantity.

For every fixed $\tau\geqslant0$, the scaling of $G_*$ yields
\begin{align*}
\lim\limits_{t\to+\infty}
t^{a_m}\|G_*(t-\tau,\cdot)-G_*(t,\cdot)\|_{L^m}=0.
\end{align*}
Indeed, let $\Phi:=G_*(1,\cdot)$,
$b_0:=\frac{n}{2\sigma_*}$, $c_0:=\frac{1}{2\sigma_*}$, and
$s:=1-\frac{\tau}{t}$. The scaling identity gives
\begin{align*}
t^{a_m}\|G_*(t-\tau,\cdot)-G_*(t,\cdot)\|_{L^m}
=\|s^{-b_0}\Phi(s^{-c_0}\,\cdot)-\Phi\|_{L^m}.
\end{align*}
For fixed $\tau$, one has $s\to1$. Continuity of dilations and scalar multiplications in $L^m$ proves convergence for $m<+\infty$; for $m=+\infty$, it follows from $\Phi\in\ml{C}_0(\mb{R}^n)$. Moreover, if
$0\leqslant\tau\leqslant\frac{t}{2}$, then
$s\in[\frac12,1]$, so the right-hand side is uniformly bounded.
Moreover, the same expression is uniformly bounded for $0\leqslant\tau\leqslant\frac{t}{2}$. Hence, the dominated convergence theorem and \eqref{Est-F-space-time} imply
\begin{align*}
\lim\limits_{t\to+\infty}t^{a_m}
\|\mb{A}_2(t)\|_{(L^m)^2}=0.
\end{align*}

To treat $\mb{A}_3$, let
\begin{align*}
\mb{h}_t(x):=\int_0^{\frac{t}{2}}\ml{F}(\tau,x)\,\mathrm{d}\tau,
\qquad
\mb{h}_{\infty}(x):=\int_0^{+\infty}\ml{F}(\tau,x)\,\mathrm{d}\tau.
\end{align*}
With $P_{\mb{h}_t}$ denoting the vector of componentwise spatial integrals, the definition of $\mb{A}_3$ can be rewritten as
\begin{align*}
\mb{A}_3(t)
=\mb{P}_0\left(
\ml{G}_*(t,|D|)\mb{h}_t
-G_*(t,\cdot)P_{\mb{h}_t}\right).
\end{align*}
Then $\mb{h}_t\to\mb{h}_{\infty}$ in $(L^1)^2$ by \eqref{Est-F-space-time}. The mass approximation used in Proposition \ref{Prop-Lm-profile}, applied componentwise to $\mb{h}_{\infty}$, gives
\begin{align*}
t^{a_m}\|\ml{G}_*(t,|D|)\mb{h}_{\infty}
-G_*(t,\cdot)P_{\mb{h}_{\infty}}\|_{(L^m)^2}=o(1),
\end{align*}
where $P_{\mb{h}_{\infty}}$ denotes the vector of the two spatial integrals. Together with
\begin{align*}
t^{a_m}\|\ml{G}_*(t,|D|)(\mb{h}_t-\mb{h}_{\infty})
-G_*(t,\cdot)P_{\mb{h}_t-\mb{h}_{\infty}}\|_{(L^m)^2}
\lesssim\|\mb{h}_t-\mb{h}_{\infty}\|_{(L^1)^2},
\end{align*}
this proves $t^{a_m}\|\mb{A}_3(t)\|_{(L^m)^2}=o(1)$.

For the fourth term, the bounded $L^m-L^m$ estimate and \eqref{Est-power-Lm} show
\begin{align*}
t^{a_m}\|\mb{A}_4(t)\|_{(L^m)^2}
&\lesssim\sum\limits_{r\in\{p,q\}}
\varepsilon^r t^{a_m}
\int_{\frac{t}{2}}^t
(1+\tau)^{-\frac{n}{2\sigma_*}(r-\frac{1}{m})}\,\mathrm{d}\tau
\|(u_0,v_0)\|_{(L^{\infty}\cap L^1)^2}^{r}\\
&\lesssim\sum\limits_{r\in\{p,q\}}
\varepsilon^r t^{1-\frac{n}{2\sigma_*}(r-1)}
\|(u_0,v_0)\|_{(L^{\infty}\cap L^1)^2}^{r}=o(1).
\end{align*}
The cancellation of the $m$-dependent exponents in the last step is
\begin{align*}
a_m+1-\frac{n}{2\sigma_*}\left(r-\frac{1}{m}\right)
=1-\frac{n}{2\sigma_*}(r-1)<0,
\end{align*}
where the strict inequality follows from
$r\geqslant\min\{p,q\}>1+\frac{2\sigma_*}{n}$.
Finally, Lemma \ref{Lemma-Lp-sharp} and \eqref{Est-F-space-time} yield
\begin{align*}
t^{a_m}\|\mb{A}_5(t)\|_{(L^m)^2}
\lesssim
\int_{\frac{t}{2}}^{+\infty}
\|\ml{F}(\tau)\|_{(L^1)^2}\,\mathrm{d}\tau=o(1).
\end{align*}
Consequently, \eqref{Dec-nonlinear-profile} gives
\begin{align}\label{Est-nonlinear-profile}
\lim\limits_{t\to+\infty}t^{a_m}
\|\ml{D}(t)-G_*(t,\cdot)\mb{P}_0\mb{M}_{\ml{F}}\|_{(L^m)^2}=0.
\end{align}

Proposition \ref{Prop-Lm-profile}, applied to the linear part, and \eqref{Est-nonlinear-profile} show
\begin{align*}
\lim\limits_{t\to+\infty}t^{a_m}
\left\|\ml{U}(t)-G_*(t,\cdot)\mb{P}_0
\left(\varepsilon\int_{\mb{R}^n}\ml{U}_0(x)\,\mathrm{d}x
+\mb{M}_{\ml{F}}\right)\right\|_{(L^m)^2}=0.
\end{align*}
Since
\begin{align*}
\mb{P}_0
\left(\varepsilon\int_{\mb{R}^n}\ml{U}_0(x)\,\mathrm{d}x
+\mb{M}_{\ml{F}}\right)
=\frac{1}{\mu+\nu}
\begin{pmatrix}\nu\\\mu\end{pmatrix}
\left(\varepsilon P_{u_0+v_0}+\ml{P}_{u^p+v^q}\right),
\end{align*}
the asserted $L^m$ profiles follow.

Taking $m=2$, using \eqref{Est-effective-L2} and the triangle inequality, for example,
\begin{align*}
\|u(t,\cdot)\|_{L^2}
&\geqslant\frac{\nu}{\mu+\nu}\|G_*(t,\cdot)\|_{L^2}
\left|\varepsilon P_{u_0+v_0}+\ml{P}_{u^p+v^q}\right|\\
&\quad-
\left\|u(t,\cdot)-\frac{\nu}{\mu+\nu}G_*(t,\cdot)
\left(\varepsilon P_{u_0+v_0}+\ml{P}_{u^p+v^q}\right)\right\|_{L^2}\\
&\gtrsim t^{-\frac{n}{4\sigma_*}}
\left|\varepsilon P_{u_0+v_0}+\ml{P}_{u^p+v^q}\right|
-o\left(t^{-\frac{n}{4\sigma_*}}\right).
\end{align*}
The estimate for $v$ is identical, and the upper bounds follow from \eqref{Est-03}. This completes the proof.

\section*{Acknowledgments}
The author thanks Wenhui Chen (Guangzhou University) for pointing out the reduction methodology and for some suggestions for Remark 2.5.


\begin{thebibliography}{99}
\bibitem{Alfaro-Ducasse-Treton=2023}
\newblock M. Alfaro, R. Ducasse, S. Tr\'eton.
\newblock The field-road diffusion model: fundamental solution and asymptotic behavior.
\newblock \emph{J. Differential Equations} \textbf{367} (2023), 332--365.
\bibitem{Aslan-Chen=2025}
\newblock H.S. Aslan, W. Chen.
\newblock On the Cauchy problem for semilinear thermoelastic plate systems in the $L^q$ framework.
\newblock \emph{Evol. Equ. Control Theory} \textbf{14} (2025), no. 6, 1565--1591.
\bibitem{Baeumer-Kovacs-Meer=2007}
\newblock B. Baeumer, M. Kov\'acs, M.M. Meerschaert.
\newblock Fractional reproduction-dispersal equations and heavy tail dispersal kernels.
\newblock \emph{Bull. Math. Biol.} \textbf{69} (2007), no. 7, 2281--2297.
\bibitem{Berestycki-Roquejoffre-Rossi=2013}
\newblock H. Berestycki, J.-M. Roquejoffre, L. Rossi.
\newblock The influence of a line with fast diffusion on Fisher-KPP propagation.
\newblock \emph{J. Math. Biol.} \textbf{66} (2013), no. 4-5, 743--766.
\bibitem{Chen=2026}
\newblock W. Chen.
\newblock Large time asymptotic behavior for the dissipative Timoshenko system and its application.
\newblock \emph{Anal. Appl. (Singap.)} \textbf{24} (2026), no. 2, 385--418.
\bibitem{Chen-Girardi=2025}
\newblock W. Chen, G. Girardi.
\newblock Sharp lifespan estimates for semilinear fractional evolution equations with critical nonlinearity.
\newblock \emph{J. Differential Equations} \textbf{443} (2025), Paper No. 113568, 37 pp.
\bibitem{Chen-Ikehata=2023}
\newblock W. Chen, R. Ikehata.
\newblock Optimal large-time estimates and singular limits for thermoelastic plate equations with the Fourier law.
\newblock \emph{Math. Methods Appl. Sci.} \textbf{46} (2023), no. 14, 14841--14866.
\bibitem{Chen-Ikehata=2024}
\newblock W. Chen, R. Ikehata.
\newblock Some remarks on large-time behaviors for the linearized compressible Navier-Stokes equations.
\newblock \emph{Differential Integral Equations} \textbf{37} (2024), no. 9-10, 699--716.
\bibitem{Chen-Ikehata=2027}
\newblock W. Chen, R. Ikehata.
\newblock Sharp large time asymptotic behavior for the multi-dimensional thermoelastic systems of type II and type III.
\newblock \emph{Commun. Contemp. Math.} (2026), Accepted. DOI: 10.1142/S0219199726500112
\bibitem{Chen-Liu=2023}
\newblock W. Chen, Y. Liu.
\newblock A note on asymptotic profiles for the thermoelastic plate system.
\newblock \emph{Proc. Amer. Math. Soc.} \textbf{151} (2023), no. 10, 4317--4329.
\bibitem{Chen-Liu=2025}
\newblock W. Chen, Y. Liu.
\newblock Large time behavior for the hyperbolic-parabolic coupled system with the regularity-loss structure.
\newblock \emph{J. Evol. Equ.} \textbf{25} (2025), no. 3, Paper No. 61, 36 pp.
\bibitem{Chen-Takeda=2023}
\newblock W. Chen, H. Takeda.
\newblock Large-time asymptotic behavior for the classical thermoelastic system.
\newblock \emph{J. Differential Equations} \textbf{377} (2023), 809--848.
\bibitem{Ebert-Girardi-Reissig=2020}
\newblock M.R. Ebert, G. Girardi, M. Reissig.
\newblock Critical regularity of nonlinearities in semilinear classical damped wave equations. \newblock \emph{Math. Ann.} \textbf{378} (2020), no. 3-4, 1311--1326.
\bibitem{Ebert-Reissig=2018-book}
\newblock M.R. Ebert, M. Reissig.
\newblock \emph{Methods for Partial Differential Equations}.
\newblock Birkh\"auser/Springer, Cham, 2018.
\bibitem{Fujita=1966}
\newblock H. Fujita.
\newblock On the blowing up of solutions of the Cauchy problem for $u_t=\Delta u+u^{1+\alpha}$.
\newblock \emph{J. Fac. Sci. Univ. Tokyo Sect. I} \textbf{13} (1966), 109--124.
\bibitem{Gal-Pohoza=2002}
\newblock V.A. Galaktionov, S.I. Pohozaev.
\newblock Existence and blow-up for higher-order semilinear parabolic equations: majorizing order-preserving operators.
\newblock \emph{Indiana Univ. Math. J.} \textbf{51} (2002), no. 6, 1321--1338.
\bibitem{Hayakawa=1973}
\newblock K. Hayakawa.
\newblock On nonexistence of global solutions of some semilinear parabolic differential equations.
\newblock \emph{Proc. Japan Acad.} \textbf{49} (1973), 503--505.
\bibitem{Hisa-Ishige=2018}
\newblock K. Hisa, K. Ishige.
\newblock Existence of solutions for a fractional semilinear parabolic equation with singular initial data.
\newblock \emph{Nonlinear Anal.} \textbf{175} (2018), 108--132.
\bibitem{Ikehata-Takeda=2017}
\newblock R. Ikehata, H. Takeda.
\newblock Critical exponent for nonlinear wave equations with frictional and viscoelastic damping terms.
\newblock \emph{Nonlinear Anal.} \textbf{148} (2017), 228--253.
\bibitem{Ikehata-Takeda=2019}
\newblock R. Ikehata, H. Takeda.
\newblock Asymptotic profiles of solutions for structural damped wave equations.
\newblock \emph{J. Dynam. Differential Equations} \textbf{31} (2019), no. 1, 537--571.
\bibitem{Ishige-Kawakami-Kobayashi=2014}
\newblock K. Ishige, T. Kawakami, K. Kobayashi.
\newblock Asymptotics for a nonlinear integral equation with a generalized heat kernel.
\newblock \emph{J. Evol. Equ.} \textbf{14} (2014), no. 4-5, 749--777.
\bibitem{Narazaki-Reissig=2013}
\newblock T. Narazaki, M. Reissig.
\newblock $L^1$ estimates for oscillating integrals related to structural damped wave models.
\newblock \emph{Studies in phase space analysis with applications to PDEs}, 215--258. Progr. Nonlinear Differential Equations Appl., \textbf{84} Birkh\"auser/Springer, New York, 2013.
\bibitem{Pang-Sun-Wang=2006}
\newblock P. Pang, F. Sun, M. Wang.
\newblock Existence and non-existence of global solutions for a higher-order semilinear parabolic system.
\newblock \emph{Indiana Univ. Math. J.} \textbf{55} (2006), no. 3, 1113--1134.
\bibitem{Sugitani=1975}
\newblock S. Sugitani. 
\newblock On nonexistence of global solutions for some nonlinear integral equations.
\newblock \emph{Osaka Math. J.} \textbf{12} (1975), 45--51.
\bibitem{Treton=2024}
\newblock S. Tr\'eton.
\newblock Blow-up vs. global existence for a Fujita-type heat exchanger system.
\newblock \emph{SIAM J. Math. Anal.} \textbf{56} (2024), no. 2, 2191--2212.
\bibitem{Wu-Tang=2015}
\newblock E. Wu, Y. Tang.
\newblock Blow-up solutions to the Cauchy problem of a fractional reaction-diffusion system.
\newblock \emph{J. Inequal. Appl.} 2015, 2015:123, 18 pp.
\end{thebibliography}
\end{document}